\documentclass[11pt, reqno]{amsart}

\usepackage{amssymb}
\usepackage[margin=2.5cm]{geometry}
\usepackage{hyperref}

\newtheorem{theorem}{Theorem}[section]
\newtheorem{lemma}[theorem]{Lemma}
\newtheorem{lem}[theorem]{Lemma}
\newtheorem{conj}[theorem]{Conjecture}
\newtheorem{prop}[theorem]{Proposition}

\newtheorem{cor}[theorem]{Corollary}
\newtheorem{remark}[theorem]{Remark}
\newtheorem{ques}[theorem]{Question}

\theoremstyle{definition}

\newcommand{\nustar}{\nu^*}

\newcommand{\gstar}{g^*}
\newcommand{\oG}{\overline{G}}
\newcommand{\eps}{\varepsilon}
\newcommand{\cupdot}{\mathbin{\mathaccent\cdot\cup}}

\numberwithin{equation}{section}

\begin{document}

\title{Many disjoint triangles in co-triangle-free graphs}



\author{Mykhaylo Tyomkyn}
\address{Department of Mathematics, California Institute of Technology, Pasadena, CA 91125, USA}
\curraddr{}
\email{tyomkyn@caltech.edu}
\thanks{}


\date{\today}

\begin{abstract}
We prove that any $n$-vertex graph whose complement is triangle-free contains $n^2/12-o(n^2)$ edge-disjoint triangles. This is tight for the disjoint union of two cliques of order $n/2$. We also prove a corresponding stability theorem, that all large graphs attaining the above bound are close to being bipartite. Our results answer a question of Alon and Linial, and make progress on a conjecture of Erd\H{o}s.
\end{abstract}
\maketitle

\section{Introduction}\label{sec:intro}

One of the classical results in extremal graph theory, Goodman's theorem~\cite{Gd}, states that in every $2$-colouring of the edges of the complete graph $K_n$ the number of monochromatic triangles is at least $\frac{1}{4}\binom{n}{3}-o(n^3)$, that is, about a quarter of all possible triangles are guaranteed to be monochromatic. 
With this in mind, Erd\H{o}s~\cite{Erd, EFGJL} asked about the number of \emph{edge-disjoint} monochromatic triangles in any $2$-colouring of $K_n$. 

To be more formal, a \emph{triangle packing} of a graph $G$ is a collection of edge-disjoint triangles in $G$. The \emph{size} of a triangle-packing is the total number of edges it contains.\footnote{This is obviously the number of the triangles in the packing times $3$. We prefer the present scaling for technical and presentation reasons.} 
Define $f(n)$ to be the largest number $m$, such that every $2$-colouring of the edges of $K_n$ contains a triangle packing of size $m$, in which each triangle is monochromatic. 

As a basic example, consider $n=6$. By the folklore fact about Ramsey numbers, any $2$-colouring of $K_6$ contains a monochromatic triangle, and it is not hard to see
that it has to contain at least two such triangles. However, they need not be edge-disjoint, as can be seen by taking a $5$-cycle and replicating a vertex. So, $f(6)=3$.
 
In general, the obvious upper bound of $f(n)\leq n^2/4-o(n^2)$ is seen to hold by considering the balanced complete bipartite graph and its complement. Erd\H{o}s~\cite{Erd, EFGJL} conjectured that this is tight.\footnote{In~\cite{Erd,EFGJL,KS} the $n^2/4+o(n^2)$ notation is used. It is understood that the additive $o(n^2)$-term can be negative, as this is the case e.g. in the above example. Hence, we believe the expression $n^2/4-o(n^2)$ better reflects the nature of the conjecture.}
\begin{conj}\label{conj:erd}
 $$f(n)=\frac{n^2}{4}-o(n^2).$$	
\end{conj}
To draw a parallel to Goodman's theorem, Conjecture~\ref{conj:erd} states that every $2$-edge-colouring of $K_n$ admits a packing with monochromatic triangles, containing about one half of all possible edges. 

In previous works, Erd\H{o}s, Faudree, Gould, Jacobson and Lehel~\cite{EFGJL} proved a first non-trivial lower bound of $f(n)\geq ({9}/{55})n^2+o(n^2)$. Keevash and Sudakov~\cite{KS} improved this, by using the fractional relaxation of the problem, to $f(n)\geq n^2/{4.3}+o(n^2)$. Alon and Linial (see ~\cite{KS}) suggested, as a step towards Conjecture~\ref{conj:erd}, to consider the natural class of colourings, in which one of the colour classes is triangle-free. 

At this stage it will be more convenient to break the symmetry and speak of a graph and its complement. A graph is said to be \emph{co-triangle-free} if its complement is triangle-free.  Equivalently, co-triangle-free graphs are graphs with independence number at most $2$. Define $g(n)$ to be the largest number $m$, such that every co-triangle-free graph on $n$ vertices contains a triangle packing of size $m$. The same example as for $f(n)$ -- the disjoint union of two cliques of order $n/2$, shows that $g(n)\leq n^2/4-o(n^2)$, and Conjecture~\ref{conj:erd} would imply that this is tight.
\begin{conj}\label{conj:AL}
	$$g(n)=\frac{n^2}{4}-o(n^2).
	$$
\end{conj} 
Yuster~\cite{Yu} worked specifically on Conjecture~\ref{conj:AL}, and proved that any potential counterexample to it must have between $0.2501 n^2$ and ${3n^2}/{8}$ edges. That is, its size cannot be too close to, or too far from the Mantel threshold. 

Our aim in this note is to give a short proof of Conjecture~\ref{conj:AL}.
\begin{theorem}\label{thm:trfree}
	We have
	$$g(n)=\frac{n^2}{4}-o(n^2).$$
\end{theorem}
Moreover, we classify the extremal graphs. An $n$-vertex graph is said to be \emph{$\eps$-far} from being bipartite if at least $\eps n^2$ edge deletions are required in order to make it bipartite.  
\begin{theorem}\label{thm:trfreestability}
For every $\eps>0$ there exists $\delta>0$ such that any co-triangle-free graph $G$ of order $n$, whose complement is $\eps$-far from being bipartite, has a triangle packing of size $(1/4+\delta) n^2+o(n^2)$.
\end{theorem}	
In other words, every co-triangle free graph on $n$ vertices admits a triangle packing on ${n^2}/{4}-o(n^2)$ edges, and the graphs achieving at most $n^2/4+o(n^2)$ are essentially co-bipartite.  

At the core of our proof is Lemma~\ref{lem:crittrfree}, which states that if a large graph $G$ is `critical', that is its complement $\oG$ is triangle-free, not bipartite, but can be made bipartite by deleting a vertex, then $G$ has a fractional triangle packing of size larger than $n(n-1)/4$. This, combined with the integer-fractional transference principle of Haxell and R\"odl (Proposition~\ref{prop:HR}), averaging over fractional packings, and a computer verification for small values of $n$ in the spirit of~\cite{KS}, yields the proof of Theorem~\ref{thm:trfree}. 

To prove Theorem~\ref{thm:trfreestability}, in addition to the above tools, we apply a theorem of Alon, Shapira and Sudakov (Proposition~\ref{prop:nearbip}) on the structure of graphs with a large edit distance to a monotone graph property. 

The rest of the paper is organized as follows. In Section~\ref{sec:prelim} we will collect some known facts about fractional and integer triangle packings. The proofs of the crucial Lemma~\ref{lem:crittrfree}, and of Theorem~\ref{thm:trfree} are carried out in Section~\ref{sec:trfree}. In Section~\ref{section:stability} we derive Theorem~\ref{thm:trfreestability}, and in  Section~\ref{sec:discuss} we discuss Conjecture~\ref{conj:erd} and related open questions.  

\section{Preliminaries}\label{sec:prelim}	

Denote by $\nu(G)$ the size of the largest triangle packing in $G$. In this notation, 
$$g(n)= \min\{\nu(G): |G|=n, \ G \text{ is co-triangle-free} \}.
$$ 
A \emph{fractional triangle packing} of $G$ is a function $w$ from $\mathcal{T}(G)$, the set of all triangles in $G$, to $[0,1]$ such that every edge $e\in E(G)$ satisfies $\sum_{T\in \mathcal{T}(G): e\subset T}w(T)\leq 1$.
The \emph{size} of a fractional packing is given by $3\sum_{T\in \mathcal{T}(G)} w(T)$. 
Define $\nustar(G)$ to be the maximum size of a fractional triangle packing of $G$; by compactness, this is well-defined. Note that ordinary triangle packings are precisely the integer-valued fractional packings --- indeed, determining $\nustar(G)$ is the LP-relaxation of the integer linear program of finding $\nu(G)$, so that $\nustar(G)\geq \nu(G)$ for every graph $G$. Consequently, we define the function $\gstar(n)$ to be the fractional counterpart to $g(n)$,
$$\gstar(n):=\min\{\nustar(G): |G|=n, \ G \text{ is co-triangle-free}\}.
$$

By the above, this function satisfies $\gstar(n)\geq g(n)$ for every $n$. On the other hand, as a consequence of the seminal theorem of Haxell and R\"odl~\cite{HR}, $\nu(G)\geq \nustar(G)-o(n^2)$ holds for every $n$-vertex graph $G$. Therefore, we have
\begin{prop}\label{prop:HR}
$$g(n)\geq \gstar(n)-o(n^2).$$	
\end{prop}  
\noindent
By virtue of Proposition~\ref{prop:HR} we can work with fractional instead of integer triangle packings at virtually no loss. Hence, going forward, the term ``packing'', unless specified otherwise, will refer to fractional triangle packings. For an $n$-vertex co-triangle-free graph $G$ define the \emph{packing density} of $G$, to be 
$$\eta(G):=\frac{\nustar(G)}{n(n-1)}.$$
It is well-known that packing densities are monotone under averaging (see e.g. Lemma 2.1 in~\cite{KS}).
\begin{lemma}\label{lem:avg}
Suppose that $G$ is a graph on $n$ vertices, and let $G_1,\dots, G_n$ be its induced subgraphs of order $n-1$. Then 
$$\eta(G)\geq \frac{1}{n}\sum_{i=1}^n\eta(G_i).
$$
\end{lemma}
\begin{proof}
Without loss of generality, assume that $V(G)=[n]$, and $V(G_i)=[n]\setminus \{i\}$.  Let $w_i$ be a packing of $G_i$ of size $\nustar(G_i)$. 
Consider $w=\frac{1}{n-2}\sum{w_i}$, which is a function on $\mathcal{T}(G)$. Any given edge $\{i,j\}$ 
contributes $0$ to $w_i+w_j$, so it receives a total weight of
\begin{equation}\label{eq:averaging}
\frac{1}{n-2}\sum_{k\neq i,j} \sum_{T=\{i,j,\ell\}\in \mathcal{T}(G_k)}w_k(T) \leq\frac{1}{n-2} \sum_{k\neq i,j} 1 =1.
\end{equation}
Thus, $w$ is a packing of $G$ of size $\frac{1}{n-2}\sum_{i=1}^n \nustar(G_i)$, which implies
$${n(n-1)\eta(G)}=\nustar(G)\geq \frac{1}{n-2}\sum_{i=1}^n \nustar(G_i)=(n-1)\sum_{i=1}^n \eta(G_i),
$$
and the desired inequality follows.
\end{proof}	
\begin{cor}\label{cor:min}
	With the above notation, $$\eta(G)\geq \min_{1\leq i\leq n} \eta(G_i).$$
	\end{cor}
We say that $G$ is \emph{co-bipartite} if its complement is bipartite. Equivalently, $G$ is co-bipartite if $V(G)$ is spanned by a disjoint union of two cliques; clearly, co-bipartite graphs are co-triangle-free. We shall need the following straightforward bound on packings of co-bipartite graphs. 
\begin{lem}\label{lem:bipartite}
	For any co-bipartite $G$ of order $n\geq 6$ we have 
	$$\nustar({G})\geq \frac{n(n-2)}{4}.
	$$
\end{lem}
\begin{proof}
$G$ contains two disjoint cliques of sizes $a$ and $n-a$, for some $0\leq a\leq n/2$. Since each clique of order  $m\geq 3$ admits a packing of size $\binom{m}{2}$, by convexity of the binomial coefficients, we have 
$$\nustar({G})\geq \binom{a}{2}+\binom{n-a}{2}\geq  2\binom{n/2}{2}=\frac{n(n-2)}{4}.
$$ 
\end{proof}
A \emph{fractional triangle decomposition} of $G$ is a packing, in which $\sum_{T\in \mathcal{T}(G): e\subset T}w(T)=1$ holds for every edge $e\in E(G)$. Fractional decompositions are packings of the largest possible size $e(G)$.

\begin{lemma}\label{lem:avgdecomp}
Suppose that $G$ is a graph on $n$ vertices, and let $G_1,\dots, G_n$ be its induced subgraphs of order $n-1$. If each $G_i$ has a fractional triangle decomposition, then so does $G$.
\end{lemma}
\begin{proof}
Assuming $V(G)=[n]$, and $V(G_i)=[n]\setminus\{i\}$, define $w$ as in the proof of Lemma~\ref{lem:avg}. We obtain~\eqref{eq:averaging} with equality in place of the inequality. Thus, $w$ is a fractional decomposition of $G$. 
\end{proof}	
 Let $K_n^{-k}$ denote the graph obtained from $K_n$ by removing a $k$-edge matching. 
\begin{lem}\label{lem:matching}
For all integers $n\geq 7$ and $0\leq k\leq \lfloor n/2\rfloor$ the graph $K_n^{-k}$ has a  fractional triangle decomposition. 
\end{lem}
\begin{proof}
It is easy to check by hand that this holds for $n=7$. The rest follows by induction, applying Lemma~\ref{lem:avgdecomp}. 
\end{proof}	

\section{Proof of Theorem~\ref{thm:trfree}}\label{sec:trfree}

Theorem~\ref{thm:trfree} follows readily from the following stability result.
\begin{lemma}\label{lem:trfree}
Suppose that $G$ is co-triangle-free, with $|G|\geq 26$ and $\eta(G) \leq {1}/{4}$. Then ${G}$ is co-bipartite. 
\end{lemma}
The reason for the threshold of $26$ is that for $|G|\leq 25$, the `natural enemy' of bipartite graphs in our problem, namely the blow-up of the $5$-cycle, achieves  $\eta \leq 1/4$. This, however, happens only for small $n$: at $n=25$, the $5$-blow-up of $C_5$ attains precisely $\eta=1/4$, and for larger $n$, as Lemma~\ref{lem:trfree} claims, only co-bipartite graphs achieve packing densities of at most $1/4$.

Let us first show that Theorem~\ref{thm:trfree} is indeed implied by Lemma~\ref{lem:trfree}.

\begin{proof}[Proof of Theorem~\ref{thm:trfree}] 
The complement of $K_{n/2,n/2}$ certifies that $g(n)\leq n^2/4-o(n^2)$. To see the other direction, suppose for a contradiction that $g(n)\leq n^2/4-\Omega(n^2)$. Then, by Proposition~\ref{prop:HR}, we have 
$$g^*(n)\leq  \frac{n^2}{4}-\Omega(n^2).$$
This means, there exists $\eps>0$ such that for large $n$ there is a co-triangle-free $G$ with $n$ vertices and $\eta(G)<1/4-\eps$. By Lemma~\ref{lem:trfree}, $G$ is co-bipartite. However, in this case, by Lemma~\ref{lem:bipartite},
$$
\nustar(G) \geq \frac{n(n-2)}{4},  
$$
so 
$\eta(G)\geq {1}/{4}-O(1/n)$, contradicting $\eta(G)<1/4-\eps$. Hence, 
$$g(n)= \frac{n^2}{4}-o(n^2).$$
\end{proof}

The proof of Lemma~\ref{lem:trfree} is carried out by induction on $n$. For both the induction base ($n=26$) and the step we require the following crucial lemma. Call a co-triangle-free graph $G$ \emph{critical} if $G$ is not co-bipartite, but contains a vertex whose removal will make it co-bipartite. 
\begin{lemma}\label{lem:crittrfree}
	Every critical graph $G$ with $|G|=n\geq 18$ satisfies
	$$\nustar(G)\geq \frac{n^2-17}{4}>\frac{n(n-1)}{4}.
	$$
	In particular,
	$$\eta(G)> \frac{1}{4}.$$
\end{lemma}
\noindent
Before giving the proof of Lemma~\ref{lem:crittrfree}, let us show how it implies Lemma~\ref{lem:trfree}.
\begin{proof}[Proof of Lemma~\ref{lem:trfree}]
We proceed by induction on $n$. The statement for $n=26$ has been computer verified via the following algorithm (the program and the execution logs are provided in supplemental files to this paper). Our code is a modification of the code from the paper of Keevash and Sudakov~\cite{KS}, tailored to meet the specific requirements of our proof.

\textbf{Initialization:} create the list $L_n$ of all triangle-free graphs on $n=6$ vertices, and calculate $\nu^*$ for their complements. 

\textbf{Iteration:} For each $n\geq 7$, go through all one-vertex triangle-free extensions of the graphs in $L_{n-1}$, 
and select from them the graphs $H$ with $\eta(\overline{H})\leq 1/4$, to form the list $L_n$. By Corollary~\ref{cor:min}, any other triangle-free graph $G$ of order $n$ must have $\eta(\oG)>1/4$.  If $L_n$ is empty, the algorithm terminates. Otherwise, move to the next iteration step.

At $n=17$, before proceeding with the iteration, delete from $L_{17}$ all bipartite graphs (be aware that this is a one-off action, which is carried out only at $n=17$). After that, perform the iteration step for $n=18$, and continue as previously. By Lemma~\ref{lem:crittrfree} and Corollary~\ref{cor:min}, for $n\geq 18$ every co-triangle-free $n$-vertex graph $G$ with $\chi(\oG)>2$ and $\eta(G)\leq 1/4$ is a one-vertex extension of an $(n-1)$-vertex graph with the same properties.
Therefore, for each $n\geq 18$ the list $L_n$ will contain precisely all triangle-free, non-bipartite $n$-vertex graphs $H$ satisfying $\eta(\overline{H})\leq 1/4$.

\textbf{Termination:} The algorithm terminates if for some $n$ the list $L_n$ is empty. 

\textbf{Outcome:} The program run terminates at $n=26$, when $L_{26}$ turns out to be empty. In fact, at $n=25$ the single graph in $L_n$, up to isomorphism, is the 5-blowup of $C_5$, and it has no valid extensions to $n=26$. This completes the proof of the induction base.

To see that Lemma~\ref{lem:crittrfree} also implies the induction step for Lemma~\ref{lem:trfree}, let $G$ be as in Lemma~\ref{lem:trfree}, with $|G|=n\geq 27$, and let $G_1,\dots,G_n$ be the induced subgraphs of $G$ of order $n-1$. By Corollary~\ref{cor:min}, we have $\eta(G_i)\leq1/4$ for some $i$, and note that $G_i$ is co-triangle-free. By the induction hypothesis, $G_i$ is co-bipartite. If $G$ is co-bipartite, we are done. Otherwise, $G$ is critical, so, by Lemma~\ref{lem:crittrfree}, we have $\eta(G)> 1/4$, a contradiction. 
\end{proof}	
\begin{remark}
Strictly speaking, the proof of Lemma~\ref{lem:trfree} uses Lemma~\ref{lem:crittrfree} only for $n\geq 27$. The latter was stated and proved for $n\geq 18$ for the purpose of accelerating the computer search needed to prove Lemma~\ref{lem:trfree} for $n=26$.
\end{remark}
\begin{proof}[Proof of Lemma~\ref{lem:crittrfree}]
Suppose that $n\geq 18$, and $G$ is a critical graph on $n$ vertices. Then there exists a vertex $v\in G$ such that $G':=\overline{G\setminus\{v\}}$ is bipartite.
Let $U\cupdot W$ be a bipartition of $V(G')$, that is $G'=G'[U,W]$, and note that the graphs $G[U]$ and $G[W]$ are complete. Note also that we can assume 
$$\min\{|U|,|W|\}\geq 7,$$
as otherwise $G[U]\cup G[W]$ would contain a packing of size more than $n(n-1)/4$, and we would be done. 
Define
\begin{align*}
	A&:=N_G(v)\cap U,\\ 
	B&:=N_G(v)\cap W,\\ 
	X&:=U\setminus A=N_{\oG}(v)\cap U, \text { and } \\ 
	Y&:=W\setminus B=N_{\oG}(v)\cap W.
\end{align*}

Note that $X$ and $Y$ are non-empty, since if, for instance, $X=\emptyset$, then $G[U\cup \{v\}]$ and $G[W]$ are complete, so $G$ would be co-bipartite, a contradiction. Moreover, since for every $(x',y')\in X\times Y$ we have $\{x',y'\}\subseteq X\cup Y\subseteq N_{\oG}(v)$, we must have $\{x',y'\}\in E(G)$, as $\oG$ is triangle-free. Hence, $G[X,Y]$ is complete bipartite.

First suppose that $|Y|$ is even (the case when $|X|$ is even is symmetric).
Let $x\in X$ be an arbitrary vertex. In the complete graph $G[Y]$ select a matching $M_Y$ on  $|Y|$ vertices, and note that $\mathcal{Y}:=\{y_1y_2x:y_1y_2\in M_Y\}$ is a triangle packing in $G$ containing $|V(M_Y)|= |Y|$ edges from $G[U,W]$. Next, let $y\in Y$ be an arbitrary vertex, and in the complete graph $G[X\setminus x]$ select a matching $M_X$ on at least $|X|-2$ vertices, so that  $\mathcal{X}:=\{x_1x_2y:x_1x_2\in M_X\}$ is a triangle packing in $G$ with $|V(M_X)|\geq |X|-2$ edges from $G[U,W]$. By construction, $\mathcal{X}$ and $\mathcal{Y}$ are edge-disjoint, and $\mathcal{X}\cup \mathcal{Y}$ contains at least $|X|+|Y|-2$ edges from $G[U,W]$. 

If both $|X|$ and $|Y|$ are odd, we select $x\in X$ arbitrarily, and $M_Y$ to be a matching on $|Y|-1$ vertices. In the second step, we select $y$ to be the sole vertex in $Y\setminus M_Y$, and $M_X$ to be a matching on $|X|-1$ vertices in $X\setminus\{x\}$. We obtain two edge-disjoint triangle packings in $G$, $\mathcal{X}$ and $\mathcal{Y}$, containing together $|X|+|Y|-2$ edges from $G[U,W]$.

Similarly, in the complete graphs $G[A]$ and $G[B]$ we select matchings $M_A$ and $M_B$, with at least $|A|-1$ and $|B|-1$ vertices, respectively, to define triangle packings $\mathcal{A}:=\{a_1a_2v : a_1a_2\in M_A\}$ and $\mathcal{B}:=\{b_1b_2v : b_1b_2\in M_B\}$. Note that $\mathcal{A}$ contains at least $|A|-1$ edges from $G[v,U]$, $\mathcal{B}$ contains at least $|B|-1$ edges from $G[v,W]$, and $\mathcal{A}, \mathcal{B},\mathcal{X}$ and $\mathcal{Y}$ are edge-disjoint.

Therefore, $\mathcal{A}\cup \mathcal{B}\cup \mathcal{X}\cup \mathcal{Y}$ is a triangle packing of $G$ containing at least 
$$|A|-1+|B|-1+|X|+|Y|-2=|U|+|W|-4=n-5$$ 
edges that are not in $G[U]$ or $G[W]$.  
The edges of $G[U]$ and $G[W]$ that are not part of $\mathcal{A}\cup \mathcal{B}\cup \mathcal{X}\cup \mathcal{Y}$ form on each of $U$ and $W$ a complete graph with a matching removed. Since $\min\{|U|,|W|\}\geq 7$, by Lemma~\ref{lem:matching} those are fractionally decomposable into triangles . Hence, 
\begin{align*}
\nustar(G)&\geq \binom{|U|}{2}+\binom{|W|}{2}+(n-5)\geq  2\binom{\frac{n-1}{2}}{2}+n-5\\
&=\frac{(n-1)(n-3)+4n-20}{4}=\frac{n^2-17}{4}>\frac{n^2-n}{4}.
\end{align*}  
In particular,
$$\eta(G)=\frac{\nustar(G)}{n(n-1)}>\frac{1}{4}.
$$
\end{proof}	

\section{Proof of Theorem~\ref{thm:trfreestability}}\label{section:stability}
For an $n$-vertex graph $G$ let $\Delta_{bip}(G)$ denote the \emph{edit distance} of $G$ to the set of bipartite graphs, i.e. the minimum number of edge deletions needed to turn $G$ into a bipartite graph. Let $E_{bip}(G):=\Delta_{bip}(G)/n^2$ be the corresponding density. So, $G$ being $\eps$-far from being bipartite  is equivalent to $E_{bip}(G)\geq\eps$. 

In order to prove Theorem~\ref{thm:trfreestability}, we need the following deep theorem of Alon, Shapira and Sudakov on monotone graph properties (\cite{AShSu}, Theorem 1.2), which we state here for the property of being bipartite.
\begin{prop}\label{prop:nearbip}{\cite{AShSu}}
For every $\eps>0$ there is $m(\eps)$ with the following property: let $G$ be any graph and suppose we randomly pick a subset $M$ on $m$ vertices from $V(G)$. Denote by $G'$ the graph induced by $G$ on $M$. Then 
$$Prob[|E_{bip}(G')-E_{bip}(G)|>\eps]<\eps.
$$  	
\end{prop}
It is implicit in~\cite{AShSu} that $m$ tends to infinity when $\eps$ goes to $0$ (in fact, it is not hard to see that this is the only way for Proposition~\ref{prop:nearbip} to be true). Thus, applying Proposition~\ref{prop:nearbip} with parameter $\eps/2$ to graphs $G$ with $E_{bip}(G)\geq\eps$, we obtain the following statement. 
\begin{cor}\label{cor:nearbip}
	For every $\eps>0$ there exists $m=m(\eps)$, with $m\rightarrow \infty$ as $\eps\rightarrow 0$, as follows. Suppose that $|G|=:n\geq m$, and $G$ is $\eps$-far from being bipartite. Then at least 
	$(1-\frac{\eps}{2})\binom{n}{m}$
	$m$-vertex induced subgraphs of $G$ are not bipartite.
\end{cor}
\begin{proof}[Proof of Theorem~\ref{thm:trfreestability}]
	Without loss of generality we may assume that $\eps<1/100$. Let $m=m(\eps)$ be as in Corollary~\ref{cor:nearbip}. By choosing $\eps$ to be sufficiently small, by Corollary~\ref{cor:nearbip} we may assume that $m>100$.
	
	By Lemma~\ref{lem:trfree} we have
	\begin{equation}\label{eq:avgavg}
	\min \{\eta(H):|H|=m,\alpha(H)\leq 2,\chi(\overline{H})> 2 \}\geq \frac{m^2-17}{4m(m-1)}=\frac{1}{4}+\frac{m-17}{4m(m-1)},
	\end{equation}
	and for co-bipartite graphs $H$ of order $m$, by Lemma~\ref{lem:bipartite}, we have
	\begin{equation}\label{eq:minusbeta}
	\eta(H)\geq \frac{m(m-2)}{4m(m-1)}=\frac{1}{4}-\frac{m}{4m(m-1)}.
	\end{equation}
	\noindent
	Suppose now that $G$ is co-triangle-free, with $|G|=n\geq m$, and $\oG$ is $\eps$-far from being bipartite.
	Applying Lemma~\ref{lem:avg} iteratively gives
	$$\eta(G)\geq\frac{1}{\binom{n}{m}}\sum_{M\in \binom{V(G)}{m}} \eta(G[M]).
	$$ 
	Combining this with Corollary~\ref{cor:nearbip}, ~\eqref{eq:avgavg} and~\eqref{eq:minusbeta}, we obtain
	\begin{align*}
	\eta(G)&\geq \frac{1}{\binom{n}{m}}\sum_{M\in \binom{V(G)}{m}} \eta(G[M])\\
	&=\frac{1}{\binom{n}{m}}\left(\sum_{M: \chi(\oG[M])>2}\eta(G[M])+
	\sum_{M: \chi(\oG[M])\leq 2}\eta(G[M])
	\right)\\ 
	&\geq (1-\frac{\eps}{2})\left(\frac{1}{4}+\frac{m-17}{4m(m-1)}\right)+
	\frac{\eps}{2}\left(\frac{1}{4}-\frac{m}{4m(m-1)}\right)\\
	&>\frac{1}{4}+\frac{m-17-\eps m+8\eps}{4m(m-1)}>\frac{1}{4}+\frac{1}{8m}.
	\end{align*}
	By the definition of $\eta$ and Proposition~\ref{prop:HR}, $$\nu(G)>(\frac{1}{4}+\frac{1}{8m})n^2+o(n^2).$$ 
	Hence, the desired statement holds with $\delta:={1}/{(8m)}$.
\end{proof}

\section{Discussion}\label{sec:discuss}

It suggest itself to use the same approach in order to tackle Conjecture~\ref{conj:erd}. Indeed, extending the definition of $\eta$ to arbitrary graphs $G$ via 
$$\eta(G):=\frac{\nustar(G)+\nustar(\oG)}{n(n-1)},
$$
the results of Section~\ref{sec:prelim} transfer straightforwardly. That said, for general graphs $\eta(G)\leq 1/4$ \emph{does not} imply that either $G$ or $\oG$ is bipartite. Take, for instance $K_{n/2,n/2}$, and add any number $\ell\leq n/8$ of edges to it. Then the largest monochromatic triangle packing in the resulting colouring $G \cup \oG$ has size at most $$2\binom{n/2}{2}+2\ell\leq \frac{n^2-2n}{4}+\frac{n}{4}=\frac{n(n-1)}{4}.$$
We suspect, however, that this is essentially the only obstruction to having $\eta(G)> 1/4$. In light of Theorem~\ref{thm:trfree}, the following strengthening of Conjecture~\ref{conj:erd} appears plausible.
\begin{conj}\label{conj:flips}
Suppose that $|G|=n\geq 26$ and $\eta(G)\leq 1/4$. Then either $G$ or $\oG$ can be made bipartite by removing at most $n/8$ edges.
\end{conj}
\noindent
The main challenge in proving Conjecture~\ref{conj:flips} is to bridge the gap between computer simulations for small $n$ and stability arguments for larger $n$. 
This seems at present much harder for general graphs than in the triangle-free case.

\subsection*{Further open problems}

As several predecessor papers~\cite{EFGJL, KS} did, we would like to draw the reader's attention to a related conjecture of Jacobson, which states that for every $n$-vertex graph $G$, one of $G$ and $\oG$ will have a triangle packing with at least $n^2/20-o(n^2)$ triangles, which is tight for the $C_5$-blowup. To prove this conjecture one would need a new idea, since the averaging approach \`a la Lemma~\ref{lem:avg} is unlikely to work.  

The works \cite{KS} and ~\cite{Yu} also discussed packings with monochromatic $k$-cliques instead of triangles. It would be interesting to study this systematically for arbitrary fixed graphs $H$, and an arbitrary number of colours.
\begin{ques}
For $c\geq 2$ and a fixed graph $H$, how many edge-disjoint monochromatic copies of $H$ are guaranteed to exist in a $c$-colouring of the edges of $K_n$?	
\end{ques}	
Specifically, it would be interesting to extend Theorem~\ref{thm:trfree} to arbitrary graphs $H$. 
\begin{ques}
How many edge-disjoint copies of $H$ are guaranteed to exist in an $n$-vertex graph whose complement is $H$-free?
\end{ques}

\section*{Acknowledgement}
I would like to thank David Conlon for helpful discussions, and Olga Goulko for help with setting up the computer simulation.

\end{document}